\documentclass[11pt, a4paper, reqno]{amsart}
\DeclareMathAlphabet{\mathscrbf}{OMS}{mdugm}{b}{n}

\usepackage[utf8]{inputenc}
\usepackage[slovene, english]{babel}
\usepackage{amsmath,amsfonts,amssymb,mathrsfs,amsthm,amscd, ulem}
\usepackage[pdftex]{graphicx}
\usepackage{tikz-cd}
\usepackage{soul} 
\usepackage{MnSymbol}
\usepackage[all,cmtip]{xy}
\usepackage{enumitem}

\usepackage{slashed}

\usepackage{pgfplots}

\usepackage[mathcal]{euscript}  

\usepackage{url}

\usetikzlibrary{arrows}

 \usepackage{relsize}

\tikzset{
  no line/.style={draw=none,
    commutative diagrams/every label/.append style={/tikz/auto=false}},
  from/.style args={#1 to #2}{to path={(#1)--(#2)\tikztonodes}}}

\title{
The Devinatz-Hopkins Theorem via  Algebraic Geometry}
\author{Rok Gregoric}
\thanks{University of Texas at Austin}
\date{\today}
\newtheorem{theorem}{Theorem}[section]
\newtheorem*{theoremm}{Theorem}

\newtheorem{corollary}[theorem]{Corollary}
\newtheorem{lemma}[theorem]{Lemma}
\newtheorem{prop}[theorem]{Proposition}

\theoremstyle{definition}
\newtheorem{definition}[theorem]{Definition}

\newtheorem{ex}[theorem]{Examples}
\newtheorem{remark}[theorem]{Remark}

\usepackage{tikz, calc}
\usetikzlibrary{matrix,arrows}

\newcommand*{\Set}{{\mathcal S\text{et}}}

\newcommand*{\Cat}{\mathcal C\mathrm{at}_\infty}

\newcommand*{\CAlg}{{\operatorname{CAlg}}}

\newcommand*{\Ab}{{\mathcal{A}\mathrm{b}}}

\newcommand*{\mX}{\mathcal X}

\newcommand*{\mS}{\mathcal S}

\newcommand*{\sO}{\mathcal O}
\newcommand*{\sF}{\mathscr F}
\newcommand*{\sE}{\mathscr E}

\newcommand*{\E}{\mathbb E_\infty}

\newcommand*{\heart}{\heartsuit}

\newcommand*{\sheafhom}{\mathscr{H}\kern -.5pt om}
\DeclareMathOperator{\Novak}{\mathscr{N}\text{\kern -3pt {\calligra\large ovak}}\,\,}

\usepackage{amsmath,calligra,mathrsfs}
\DeclareMathOperator{\fHom}{\mathscr{H}\text{\kern -3pt {\calligra\large om}}\,}

\DeclareMathOperator{\Hom}{\operatorname{Hom}}

\DeclareMathOperator{\Sp}{\operatorname{Sp}}
\DeclareMathOperator{\Zp}{\operatorname{Sp}}
\DeclareMathOperator{\Spcn}{\operatorname{Sp^{cn}}}

\DeclareMathOperator{\Fun}{\operatorname{Fun}\,}

\DeclareMathOperator{\Spec}{\operatorname{Spec}}
\DeclareMathOperator{\Spf}{\operatorname{Spf}}
\DeclareMathOperator{\Zpf}{\operatorname{Spf}}

\DeclareMathOperator{\Map}{\operatorname{Map}}
\DeclareMathOperator{\QCoh}{\operatorname{QCoh}}

\DeclareMathOperator{\Tot}{\operatorname{Tot}}

\DeclareMathOperator{\G}{\mathbf G}

\DeclareMathOperator{\CMon}{\operatorname{CMon}}

\DeclareMathOperator{\Mod}{\operatorname{Mod}}

\renewcommand{\i}{\infty}

\renewcommand{\Pr}{\mathcal P\mathrm r}

\newcommand{\w}{\widehat}

\renewcommand{\i}{\infty}

\DeclareFontFamily{U}{matha}{\hyphenchar\font45}
\DeclareFontShape{U}{matha}{m}{n}{
      <5> <6> <7> <8> <9> <10> gen * matha
      <10.95> matha10 <12> <14.4> <17.28> <20.74> <24.88> matha12
      }{}
\DeclareSymbolFont{matha}{U}{matha}{m}{n}

\DeclareMathSymbol{\varsubset}{3}{matha}{"80}

\usepackage{color}   
\usepackage{hyperref}
\hypersetup{
    colorlinks=true, 
    linktoc=all,     
    linkcolor=black,  
}

\renewcommand{\i}{\infty}

\begin{document}
\maketitle

\begin{abstract}
In this note, we show how a continuous action of the Morava stabilizer group $\mathbb G_n$ on the Lubin-Tate spectrum $E_n$,  satisfying the conclusion $E_n^{h\mathbb G_n}\simeq L_{K(n)} S$ of the Devinatz-Hopkins Theorem, may be obtained by   monodromy on the stack of oriented deformations of formal groups in the context of formal spectral algebraic geometry. 
\end{abstract}

A classical and computationally invaluable result in chromatic homotopy theory, the Morava Change of Rings Theorem,  see for instance \cite{DevMorava}, identifies the second page of the $K(n)$-local Adams spectral sequence for the Lubin-Tate spectrum $E_n$ as continuous group cohomology
$$
E^{s, t}_2\simeq \mathrm H^s_\mathrm{cont}(\mathbb G_n; \pi_t(E_n)) \Rightarrow \pi_{t-s}(L_{K(n)}S).
$$
A conceptual spectrum-level explanation for this isomorphism is given by the Devinatz-Hopkins Theorem of \cite{DevHop}. It asserts the existence of a (suitably-interpreted) continuous action of the Morava stabilizer group $\mathbb G_n$ on the Lubin-Tate spectrum $E_n$, such that its continuous homotopy fixed points  are
\begin{equation}\label{The Theorem}
E_n^{h\mathbb G_n}\simeq L_{K(n)}S.
\end{equation}
The proof of the equivalence \eqref{The Theorem} has by now become largely standard, using nilpotence technology applied to the $K(n)$-local Amitsur complex of $E_n$, and ultimately stemming from the key observation of Hopkins \textit{et.\,al.}\,that the Adams spectral sequence of $E_n$ possesses a horizontal vanishing line. The somewhat less straightforward part is instead identifying said Amitsur complex with the simplicial bar resolution of a suitably-interpreted continuous action of $\mathbb G_n$ on $E_n$. That was accomplished in a somewhat \textit{ad hoc} manner in \cite{DevHop}, and in various contexts of continuous group actions of spectra such as \cite{Davies}, \cite{Quick} (though these approaches ostensibly amount to enriching the construction from \cite{DevHop}). A formalization using the condensed set technology of \cite{Cond} to tackle continuity has also been announced by Clausen-Scholze.

In contrast, we propose to side-step the issue of continuous actions altogether. Instead, we exhibit the action in an appropriate context of formal spectral algebraic geometry.
Our results may be summarized as follows.

\begin{theoremm}
The Morava stabilizer group $\mathbb G_n$ admits a canonical action on the formal spectral stack $\Spf (E_n)$. For 
continuous homotopy fixed points of this action defined as  $E^{h\mathbb G_n}_n := \sO(\Spf(E_n)/\mathbb G_n)$, there is a canonical equivalence $E^{h\mathbb G_n}_n\simeq L_{K(n)}S.$ Furthermore, the three resulting spectral sequences coincide:
\begin{enumerate}
\item The descent spectral sequence for the structure sheaf on $\Spf(E_n)/\mathbb G_n$
$$
E^{s, t}_2 = \mathrm H^s(\Spf(E_n)/\mathbb G_n; \pi_t(\sO))\Rightarrow \pi_{t-s}(L_{K(n)}S).
$$
\item The homotopy fixed point spectral sequence for the $\mathbb G_n$-action on $E_n$
$$
E^{s, t}_2 =\mathrm H^s(\mathrm B\mathbb G_n;\pi_t(E_n))\Rightarrow \pi_{t-s}(L_{K(n)}S).
$$
\item The $K(n)$-local Adams spectral sequence for $E_n$
$$
E_2^{s,t}= \mathrm{Ext}^{s,t}_{\pi_*(L_{K(n)}(E_n\otimes E_n))}(\pi_*(E_n), \pi_*(E_n))\Rightarrow \pi_{t-s}(L_{K(n)}S).
$$
\end{enumerate} 
\end{theoremm}

Our approach is based on a theorem of Lurie \cite[Theorem 5.1.5]{Elliptic 2}, identifying $\Spf (E_n)$ with the moduli stack of oriented deformation of a height $n$ formal group. We show that the Morava stabilizer group action arises as an instance of monodromy actions on de Rham spaces.
To establish the above version of the Devinatz-Hopkins Theorem in our setting, we employ similar arguments to the analogous considerations in classical formal algebraic geometry from \cite[Chapter 7]{Goerss}.

The computational underpinning of our proof (somewhat obscured in our account) is the fundamental observation from \cite{DevHop} that the $K(n)$-local Adams spectral sequence for $E_n$ possesses a horizontal vanishing line. Ours is in that sense analogous to all of the currently known approaches to the Devinatz-Hopkins Theorem, including, to the best of the author's understanding, the forthcoming work of Clausen-Scholze. The latter construct the continuous (or in their setting more precisely: condensed) Morava stabilizer group action  similarly to us, in that they employ results\footnote{Though unlike our account, where the algebro-geometric aspect of the results in \cite{Elliptic 2} are center-stage, the approach of Clausen-Scholze only relies on the more flexible functoriality of Lubin-Tate theory (in particular: that its base can be taken to be an arbitrary perfect $\mathbf F_p$-algebra as base, as opposed to only a perfect field) afforded by Lurie's construction, as compared to the traditional one by Goerss-Hopkins-Miller.} from \cite{Elliptic 2}. In particular, we wish to make it clear that the majority of our proof of the Devinatz-Hopkins Theorem follows the same reasoning and insights as  the original account in \cite{DevHop}. 

Our primary contribution is a novel way to obtain the Morava stabilizer group action by way of formal spectral algebraic geometry, building on Lurie's  work in \cite{Elliptic 2} and \cite{SAG}. Related applications of those results to topics in chromatic homotopy theory, primarily concerning Gross-Hopkins duality, are considered in \cite{Zanath}.

\subsection*{Acknowledgments}

I would like to thank Andrew Blumberg and David Ben-Zvi, without whose constant support and encouragement this note would surely have never come to be. Thanks also to Ben Antieau, David Gepner, Nat Stapleton, and especially Paul Goerss, for offering helpful comments on the draft. And finally, I am grateful to Agnès Beaudry, Mike Hill, Markus Pflaum, and Dylan Wilson, for organizing \textit{Chromatic Homotopy: Journey to the Frontier} at UC Boulder in 2018, where I had my first chance to really breathe in the fresh air of chromatic homotopy theory.

\section{Background on formal spectral algebraic geometry}

We begin by summarizing some notions and results from \cite{SAG} and \cite{Elliptic 2} which are key for the purpose of this note.

\subsection{Adic $\E$-rings and formal SAG}
From the functor of points perspective, formal spectral algebraic geometry, in the form relevant to us and in \cite{Elliptic 2} (but slightly differently from \cite[Definition 8.1.1.5]{SAG}, where a connectivity assumption is imposed throughout; see \cite[\textit{Nonconnective Ring Spectra} on page 13]{SAG}), concerns functors $\CAlg^\mathrm{ad}_\mathrm{cpl}\to\mS$.

Here $\CAlg^\mathrm{ad}_\mathrm{cpl}$ denotes the $\i$-category of \textit{complete adic $\E$-rings} in the sense of \cite[Definition 0.0.11]{Elliptic 2}. That is, an object of it consists of an $\E$-ring $A$, together with a topology on $\pi_0(A)$ which admits a finitely generated ideal of definition $I\subseteq \pi_0(A)$, such that the topology on $\pi_0(A)$ is equivalent to the $I$-adic topology, and finally such that the $\E$-ring $A$ is $I$-complete in the sense of \cite[Definition 7.2.3.22]{HA}. Given such a complete adic $\E$-ring $A$, we define its \textit{formal spectrum} to be the corepresentable functor $\Spf(A) : \CAlg^\mathrm{ad}_\mathrm{cpl}\to\mS$ given by
$$
B\mapsto \Map^\mathrm{cont}_{\CAlg}(A, B):= \Map_{\CAlg}(A, B)\times_{\Hom_{\CAlg^\heart}\left(\pi_0(A), \pi_0(B)\right)}\Hom^\mathrm{cont}_{\CAlg^\heart}\left(A, B\right).
$$
Of course, the embedding $\CAlg^\mathrm{ad}_\mathrm{cpl}\to \Fun(\CAlg^\mathrm{ad}_\mathrm{cpl}, \mS)$ is fully faithful, and its codomain is a convenient place to do formal spectral algebraic geometry.

\subsection{Formal groups over $\E$-rings}
As an instance of that motto, the theory of formal groups over $\E$-rings is developed in \cite[Chapter 1]{Elliptic 2}. We give a slightly informal account, and refer to \textit{loc.\,cit.}~for a precise and detailed account.

\begin{definition}\label{def of fg}
A \textit{formal group over an $\E$-ring $A$} is an abelian group object in the $\infty$-category of 1-dimensional  fiber-smooth formal spectral $A$-schemes.
\end{definition}

\begin{remark}\label{remarkee}
There are a number of caveats concerning the above definition:
\begin{enumerate}[label =(\arabic*)]
\item The notion of an \textit{abelian group object} must  be understood in the sense of \cite[Section 1.2]{Elliptic 1}. That is to say, we must equip its Yoneda presheaf with a factorization through the functor
 $\Omega^\infty:\Mod^{\mathrm{cn}}_\mathbf Z\to \mS$, or equivalently, the
 forgetful functor $\mathcal{T}\mathrm{op}\mathcal A\mathrm{b}\to \mS$.  This is a \textit{strictified} version of the more familiar notion of a grouplike $\E$-algebra objects, since the Yoneda presheaf is in the latter case asked to factor through $\Omega^\infty :\Spcn\to \mS$, or equivalently, the forgetful functor $\CMon^\mathrm{gp}(\mS)\to\mS$.

\item 
The requirement of \textit{fiber-smoothness} on a formal $A$-scheme $X$ is taken in the sense of \cite[Definition 11.2.3.1]{SAG}, and roughly amounts to asking for $X$ to be locally isomorphic to the formal affine line  $\widehat{\mathbf A}^1_A=\Spf (A[\![t]\!])$. In particular, this implies that $X$ is a flat over $A$. This differs from the notion of differential smoothness in the sense of \cite[Definition 11.2.2.2]{SAG}, which imposes conditions on the cotangent complex $L_{X/A}$, but is incompatible with flatness unless $A$ is a $\mathbf Q$-algebra. Since we want ordinary formal groups over commutative rings to be special cases of Definition \ref{def of fg}, and they are indeed flat, we therefore have no choice but to use fiber-smoothness instead of differential smoothness.

\item \label{moltres}
Definition \ref{def of fg} is really only correct when the $\E$-ring $A$ is connective. For a non-connective $\E$-ring $A$, we should instead define formal groups over $A$ to be formal groups in the above sense over the connective cover $\tau_{\ge 0}(A)$. However, certain constructions associated to a formal group $\w{\G}$, for instance the $\E$-algebra of functions $\sO_{\w{\G}}$, depend on whether we are considering it as existing over $A$ or over $\tau_{\ge 0}(A)$.
\end{enumerate}
\end{remark}

\begin{ex}
 The following are the only classes of formal groups that we will be concerned with in this note:
\begin{itemize}
\item 
Over a commutative ring $A$, viewed as a discrete $\E$-ring, Definition \ref{def of fg}
reproduces the usual meaning of (as always, $1$-dimensional smooth) formal groups over $A$. 
\item Let $A$ be a \textit{complex periodic} $\E$-ring, i.e.\ complex orientable and $\pi_2(A)$ is a locally free $\pi_0(A)$ module of rank $1$. Then the \textit{Quillen formal group of $A$} is
$$\w{\G}{}^{\CMcal Q}_A :=\Spf (C^*(\mathbf{CP}^\infty; A)),$$
which indeed gives rise to a formal group over $A$ by \cite[Subsection 4.1.3]{Elliptic 2}.
\end{itemize}
\end{ex}

Formal groups over $A$ form an $\i$-category $\mathcal M_\mathrm{FG}(A)$, and this constuction is functorial in $A$ by base-change:

\begin{definition}\label{pullback}
Let $f:A\to B$ be a map of $\E$-rings, and $\w{\G}$ a formal group over $A$. The 
 pullback of formal spectral schemes along $\Spec(f):\Spec (B)\to \Spec (A)$ gives rise to a formal group over $B$, which we denote $f^*\widehat{\G}$.
\end{definition}

There is also another slightly different form of functoriality afforded to formal groups. Sending
$$\w{\G}\mapsto \w{\G}{}^0 :=\Spf(\pi_0(\sO_{\w{\G}}))$$
gives rise to a functor $\mathcal M_\mathrm{FG}(A)\to \mathcal M_\mathrm{FG}(\pi_0(A))$. Informally, this sends a spectral formal group to its underlying ordinary formal group.

\begin{remark}
When the $\E$-ring $A$ is connective, the preceding construction is a special case of Definition \ref{pullback}. Indeed, in that case there exists a map of $\E$-rings $t :A\to \pi_0(A)$, and $\w{\G}{}^0\simeq t^*\w{\G}$. For a non-connective $\E$-ring $A$ on the other hand, the connection between $A$ and $\pi_0(A)$ is only through the span $A\leftarrow \tau_{\ge 0}(A)\to \pi_0(A)$, and so $\w{\G}\mapsto \w{\G}{}^0$ is not merely an instance of base-change. This is closely related to the subtleties alluded to in item \ref{moltres} of Remark \ref{remarkee}.
\end{remark}

\subsection{Orientations and deformations of formal groups}
The class of formal groups singled out by the following definition is of special importance in relation to chromatic homotopy theory. Here an $\mathbb E_\infty$-ring $A$ is called \textit{complex periodic}  \cite[Definition 4.1.8]{SAG} if it is both complex orientable and weakly 2-periodic.
\begin{definition}[{\cite[Proposition 4.3.23]{Elliptic 2}}]
A formal group $\w{\G}$ over an $\E$-ring $A$ is \textit{oriented} if and only if
 $A$ is complex periodic and 
 $\w{\G}\simeq \w{\G}{}^{\CMcal Q}_A$
is its Quillen formal group. We denote by $\mathcal M_{\mathrm{FG}}^\mathrm{or}(A)\subseteq\mathcal M_{\mathrm{FG}}(A)$ the subspace of oriented formal groups over $A$.
\end{definition}

\begin{remark}
Though the above form is the most practical for our purposes, we would be remiss not to summarize an equivalent but better-motivated approach to defining oriented formal groups \cite[Definition 4.3.9]{Elliptic 2}.
To any formal group $\w{\G}$ over an $\E$-ring $A$ we may by \cite[Subsections 5.2.1 - 5.2.3]{Elliptic 2} associate an $A$-module $\omega_{\w{\G}}$, its \textit{dualizing line}, and the analogue of the module of invariant differentials on a classical formal group. An \textit{orientation of $\w{\G}$} then amounts to an $A$-linear equivalence $\omega_{\w{\G}}\simeq \Sigma^{-2}(A)$.
This is in spirit a $2$-shifted analogue of the various notions of orientation in
 classical geometric contexts, where it usually means some kind of trivialization of a bundle of volume forms.
\end{remark}

The space of \textit{deformations of $\w{\G}_0$ over $A$} is defined as
$$
\mathrm{Def}_{\w{\G}_0}(A) :=\varinjlim_I \mathrm{Hom}_{\CAlg^\heart}(\kappa, \pi_0(A)/I)\times_{\mathcal M_\mathrm{FG}(\pi_0(A)/I)}\mathcal M_\mathrm{FG}(A),
$$
with the colimit ranging over all the ideals of definition $I\subseteq\pi_0(A)$. Informally, this consists of a ring homomorphism $f:\kappa\to \pi_0(A)/I$, a formal group $\w{\G}$ over the $\E$-ring $A$, and an isomorphism $f^*\w{\G}_0\simeq q^*\w{\G}{}^0$ of formal groups over $\pi_0(A)/I$, where $q : \pi_0(A)\to \pi_0(A)/I$ is the quotient projection. \textit{Oriented deformations} are defined analogously as
$$
\mathrm{Def}_{\w{\G}_0}(A) :=\varinjlim_I \mathrm{Hom}_{\CAlg^\heart}(\kappa, \pi_0(A)/I)\times_{\mathcal M_\mathrm{FG}(\pi_0(A)/I)}\mathcal M_\mathrm{FG}^\mathrm{or}(A).
$$
Both of these construction respect pullback along maps of adic $\E$-rings, and as such give rise to  functors $\mathrm{Def}_{\w{\G}_0}, \mathrm{Def}^\mathrm{or}_{\w{\G}_0}: \CAlg^\mathrm{ad}_\mathrm{cpl}\to \mS$.

 The following theorem of Lurie, a  cousin of the Goerss-Hopkins-Miller Theorem, may be taken as the definition of Lubin-Tate spectra, and is the bedrock of this note.

\begin{theorem}[{\cite[Theorem 5.1.5, Remark 6.0.7]{Elliptic 2}}]
Let $\w{\G}_0$ be a formal group of finite height over a perfect field  $\kappa$ of characteristic $p>0$. 
Let $E(\kappa, \w{\G})$ be the Lubin-Tate spectrum of $\w{\G}_0$, viewed as an adic $\E$-ring with respect to the $n$-th Landweber ideal $\mathfrak I_n\subseteq\pi_0(E(\kappa, \w{\G}_0)).$ There is a natural equivalence
$$
\Spf(E(\kappa, \w{\G}_0))\simeq \mathrm{Def}^\mathrm{or}_{\w{\G}_0}
$$
in the $\i$-category $\Fun(\CAlg^\mathrm{ad}_\mathrm{cpl}, \mS)$.
\end{theorem}

\begin{remark}
Lurie formulates his result (which also works over more general perfect base rings than a field) in terms of deformations of $p$-divisible groups instead of formal groups. This has the advantage of being more general, applying for instance also to \'etale $p$-divisible groups, and is crucial in the follow-up paper \cite{Elliptic 3} on Hopkins-Kuhn-Ravenel character theory and transchromatic ambidxterity. Alas, for our purposes, since all the $p$-divisible groups in sight would be connected, the analogue of Tate's Theorem in \cite[Section 2.3]{Elliptic 2}, allows us to restrict to formal groups instead. Ultimately however, this is nothing more than an aesthetic preference, and this note could well have been written with the functor $\mathcal M_\mathrm{BT}$ everywhere in place of $\mathcal M_\mathrm{FG}$.
\end{remark}

\section{Morava stabilizer group action and fixed points}
\subsection{Complete Notherian local $\E$-rings}
For the remainder of this note, $\kappa$ will be a perfect field of characteristic $p>0$. We find it convenient to restrict to a smaller subcategory of $\CAlg^\mathrm{ad}_\mathrm{cpl}$, consisting roughly of complete Notherian local $\E$-rings with residue field $\kappa$.

\begin{definition} 
Let  $\CAlg^{\mathrm{cN}}_{/\kappa}\subseteq \CAlg^{\mathrm{ad}}_\mathrm{cpl}$ denote the subcategory spanned by complete adic $\E$-rings $A$ for which the commutative ring $\pi_0(A)$ is a local Noetherian ring with maximal ideal $\mathfrak m$, topologized with respect to the $\mathfrak m$-adic topology, and such that there exists an abstract  (i.e.\ non-specified) isomorphism
 $\pi_0(A)/\mathfrak m\simeq \kappa$.
\end{definition}

\begin{remark}
The notation  $\CAlg^{\mathrm{cN}}_{/\kappa}$ is potentially  misleading.
Indeed, unlike what it my seem to indicate, said $\i$-category \textit{is not} equivalent to a subcategory of the overcategory $\CAlg_{/\kappa}$. That would only hold if we restricted to connective objects on both sides, while our primary interest rests with the non-connective complex periodic $\E$-rings. Similarly, connective objects in $\CAlg^{\mathrm{cN}}_{/\kappa}$ are not Noetherian $\E$-rings in the sense of \cite[Definition 7.2.4.30]{HA}. But we could not have used that notion of Noetherianness in the above definition, since it again only applies to connective $\E$-rings.
\end{remark}

From here on, we will consider the $\i$-category $\Fun(\CAlg^{\mathrm{cN}}_{/\kappa}, \mS)$ as the setting for formal spectral algebraic geometry.
In particular, we will implicitly restrict the domain of the functor $\Spf (A)$ to the subcategory $\CAlg^\mathrm{cN}_{/\kappa}\subseteq\CAlg^\mathrm{ad}_\mathrm{cpl}$ for any adic $\E$-ring $A$.

The functors of the \textit{ring of functions} $\sO : \Fun(\CAlg^{\mathrm{cN}}_{/\kappa}, \mS)\to \CAlg^\mathrm{ad}_\mathrm{cpl}$ and the $\i$-category of \textit{quasi-coherent sheaves} $\QCoh : \Fun(\CAlg^{\mathrm{cN}}_{/\kappa}, \mS)\to\Cat$ are defined by Kan extension from the subcategory of affines (i.e.\ representable functors) on which they are defined as
$$
\sO(\Zpf (A)) := A,\qquad\quad \QCoh(\Zpf(A)):=\Mod^\mathrm{cpl}_A.
$$
Noting that we may have equivalently replaced the $\Cat$ with $\Pr^\mathrm L$ in the definition of quasi-coherent sheaves, we see that any map of functors $f:X\to Y$ induces adjoint functors $f^* :\QCoh(Y)\rightleftarrows \QCoh(X) : f_*$, the familiar pullback and pushforward\footnote{As usual however, pushforward may not be very well-behaved without some basic additional assumptions.} functoriality. In particular, we call pushforward along the terminal map $p : X\to *$ \textit{global sections} and denote $\Gamma(X;\sF) :=p_*(\sF)$ for any $\sF\in\QCoh(X)$. For the structure sheaf $\sF = \sO_X$, global sections $\Gamma(X; \sO_X)\simeq \sO(X)$ recover the ring of functions.

\begin{remark}
Because we are not equipping $\CAlg^{\mathrm{cN}}_{/\kappa}$ with a Grothendieck topology, questions of descent are beyond our reach. Fortunately, as explained for $\QCoh$ in \cite[Proposition 6.2.3.1]{SAG} (in only a slightly different setting), both $\sO$ and $\QCoh$ are agnostic regarding sheafification, making their definition unambiguous.
\end{remark}

\subsection{A short digression on monodromy}
In the proof of Proposition \ref{King of Kings} in the next Subsection, we will need a certain result, which becomes particularly simple and natural when viewed in a slightly more general context than strictly necessary for our purposes. 

Recall that monodromy is classically understood to be the action of the fundamental group $\pi_1(X, x)$ of a base-space $X$ on the fiber $\mathcal L_x$ of a local systems $\mathcal L$ on $X$, acting through parallel transport around loops. The following is a simple incarnation of that idea in the setting of an $\i$-topos, but with the notion of a ``point" being understood in the generalized sense of algebraic geometry.

\begin{lemma}\label{Monodromy}
Let $x:P\to X$ be a morphism in an $\i$-topos $\mX$.
\begin{enumerate}[label =(\roman*)]
\item The ``based loop space" $\Omega_x(X) := P\times_X P$ admits a canonical group structure in the overtopos $\mX_{/P}$, exhibiting it as an object $\Omega_x(X)\in\mathrm{Grp}(\mX_{/P})$. There is a canonical equivalence of simplicial objects\label{un}
$$
\mathrm B^\bullet_{\Omega_x(X)}(P, P)\simeq \mathrm{\check{C}}^\bullet(P\xrightarrow{x}X)
$$
between its bar construction in $\mX_{/P}$ and the \v{C}ech nerve of $x$ in $\mX$.

\item For any object  $Y\in \mX{}_{/X}$, we define its ``fiber over $x$" through the pullback square
\begin{equation}\label{The Square}
\xymatrix{
x^*(Y)\ar[r]\ar[d] & Y\ar[d]\\
P\ar[r]^x & X
}
\end{equation}
 in $\mX$.
This ``fiber" $x^*(Y)\in \mX{}_{/P}$ admits a canonical $\Omega_x(X)$-action, whose bar constriction in $\mX{}_{/P}$ is equivalent to the \v{C}ech nerve in $\mX$ \label{deux}
$$
	\mathrm B^\bullet_{\Omega_x(X)}(P, x^*(Y))\simeq \mathrm{\check{C}}^\bullet(x^*(Y)\to Y).
$$
\end{enumerate}
\end{lemma}

\begin{proof}
Recall from \cite[Subsection 6.1.2]{HTT} and \cite[Proposition 2.4.2.5]{HA} that group objects and group actions (or their common generalization, groupoid objects) in an $\i$-topos are completely and equivalently encoded by their bar constructions. Thus it is necessary and sufficient to verify that the \v{C}ech complexes in question are of the appropriate forms for a group object and group action respectively.

For \ref{un}, we rewrite the \v{C}ech complex of the morphism  $x$ as
$$
\mathrm{\check{C}}^\bullet(x)\simeq \underbrace{P\times_X\cdots\times_X P}_{\bullet +1}\simeq \underbrace{(P\times_X P)\times_P\cdots \times_P (P\times_X P)}_\bullet\simeq \underbrace{\Omega_x(X)\times_P\cdots\times_P\Omega_x(X)}_{\bullet}.
$$
It follows clearly that it satisfies the Segal condition and exhibits $\Omega_x(X)\in \mathrm{Grp}(\mX_{/P})$.

For \ref{deux}, observe that we may compare the two \v{C}ech nerves in sight via (degree-wise) pullback of simplicial objects. Combining that with point \ref{un}, we get equivalences of simplicial objects
\begin{eqnarray*}
\mathrm{\check{C}}^\bullet(x^*(Y)\to Y)
&\simeq& \mathrm{\check{C}}^\bullet(P\to X)\times_X Y\\
&\simeq &\mathrm B^\bullet_{\Omega_x(X)}(P, P)\times_X Y\\
&\simeq &\mathrm B^\bullet_{\Omega_x(X)}(P, P\times_X Y)\\
&\simeq &\mathrm B^\bullet_{\Omega_x(X)}(P, x^*(Y)),
\end{eqnarray*}
 exhibiting the desired $\Omega_x(X)$-action on the fiber $x^*(Y)$.
\end{proof}

\begin{remark}
In the setting of Lemma \ref{Monodromy}, 
passage to geometric realizations from \ref{un} gives an equivalence $\mathrm B\Omega_x(X)\simeq X^\wedge_x$ between the classifying space for $\Omega_x(X)$ (in the overtopos $\mX_{/P}$) and the so-called nilpotent completion of $X$ at $x$, defined as $X^\wedge_x =|\check{\mathrm C}^\bullet(x)|$. When $x:P\to X$ is an effective epimorphism, then $X^\wedge_x\simeq X$. Then Lemma \ref{Monodromy} \ref{deux} shows that $Y\simeq x^*(Y)/\Omega_x(X)$, generalizing the classical fact that a local system on a connected base-space is completely determined by its monodromy representation.
\end{remark}

\begin{remark}\label{remarques} Let us take for $\mX$ the presheaf $\i$-topos $\Fun(\CAlg^\mathrm{cn},\mS)$, the usual setting for ``functor of points"  spectral algebraic geometry (once again ignoring questions of descent). A connective $\E$-ring $A$ gives rise to the terminal map $x_A:\Spec(A)\to\Spec(A)$, which we may view as an $A$-point of $\Spec(S)$. It follows from Lemma \ref{Monodromy}, \ref{un} that the loop space $\Omega_{x_A}(\Spec(S))$ admits a group structure over $\Spec(A)$. That amounts to an appropriately-interpreted (see \cite{Torii} for a thorough discussion of appropriate coalgebras in this setting) Hopf algebroid structure on $\sO(\Omega_{x_A}(\Spec (S))\simeq A\otimes A$ over $A$. Upon passage to homotopy groups, this recovers the usual ``generalized dual Steenrod algebra" Hopf algebroid $(\pi_*(A), \pi_*(A\otimes A)) = (A_*, A_*A)$. Similarly, given any connective $\E$-ring $A$, the $\Omega_{x_A}(\Spec(S))$-action on the fiber $x_A^*(\Spec(X))$ described in Lemma \ref{Monodromy}, \ref{deux}, gives rise on homotopy groups to the usual ``generalized Steenrod comodule" structure on $\pi_*(A\otimes X) = A_*(X)$. This hints at the relationship between the monodromy construction of Lemma \ref{Monodromy} and generalized Adams spectral sequences, which we partly elucidate in Subsection \ref{Adams section}, and in Remark \ref{ANZZ for real} in a bit more detail in the case of the Adams-Novikov spectral sequence.
\end{remark}

\subsection{Morava stabilizer group action on oriented deformations}

Fix a formal group $\w{\G}_0$ of finite height over $\kappa =\overline{\mathbf F}_p$ and let $\mathbb G(\kappa, \w{\G}_0)$ be its (big, i.e.\ extended) Morava stabilizer group, viewed as an algebraic group, and hence a functor $\CAlg^\mathrm{cN}_{/\kappa}\to \mS$, as explained in \cite[Remark 5.29]{Goerss}. 

\begin{prop}\label{King of Kings}
There exists a canonical action of the Morava stabilizer group $\mathbb G(\kappa,\w{\G}_0)$ on the oriented deformations $\mathrm{Def}^\mathrm{or}_{\w{\G}_0}$ in $\mathrm{Fun}(\CAlg^\mathrm{cN}_{/\kappa}, \mS)$, whose two-sided bar construction is equivalent as a simplicial object in $\Fun(\CAlg^\mathrm{cN}_{/\kappa}, \mS)$
$$ \mathrm{\check{C}}^\bullet(\mathrm{Def}^\mathrm{or}_{\w{\G}_0}\to *)\simeq \mathrm B^\bullet_{\mathbb G(\kappa, \w{\G}_0)}(*, \mathrm{Def}^\mathrm{or}_{\w{\G}_0}) $$
to the \v{C}ech nerve of (the terminal map of) $\mathrm{Def}^\mathrm{or}_{\w{\G}_0}$.
\end{prop}

\begin{proof}[Proof]
By the definition of oriented deformations, we have
\begin{equation}\label{useful}
\mathrm{Def}_{\w{\G}_0}^\mathrm{or}\simeq \mathrm{Def}_{\w{\G}_0}\times_{\mathcal M_{\mathrm{FG}}}\mathcal M_{\mathrm{FG}}^\mathrm{or}.
\end{equation}
The factor $\mathcal M_\mathrm{FG}^\mathrm{or}$ in this fibered product may be replaced with $\{\w{\G}{}^{\CMcal Q}_A\}$ when $A$ is complex oriented, and with $\emptyset$ when $A$ is not. It follows from this observation that
$$
\mathrm{\check{C}}^\bullet(\mathrm{Def}^\mathrm{or}_{\w{\G}_0}\to *)\simeq \mathrm{\check{C}}^\bullet(\mathrm{Def}_{\w{\G}_0}\to\mathcal M_{\mathrm{FG}})\times_{\mathcal M_\mathrm{FG}}\mathcal M^\mathrm{or}_\mathrm{FG}.
$$
as the base-change of simplicial objects. Consequently, pulling back the equivalence of simplicial objects from the next Lemma \ref{Action} along the inclusion $\mathcal M_\mathrm{FG}\to \mathcal M_\mathrm{FG}^\mathrm{or}$ gives rise to a $\mathbb G(\kappa, \w{\G}_0)$-action on $\mathrm {Def}^\mathrm{or}_{\w{\G}_0}$ with the desired bar construction.
\end{proof}

\begin{lemma}\label{Action}
There exists a canonical action of the Morava stabilizer group $\mathbb G(\kappa,\w{\G}_0)$ on the unoriented deformations $\mathrm{Def}_{\w{\G}_0}$ in $\mathrm{Fun}(\CAlg^\mathrm{cN}_{/\kappa}, \mS)$, whose two-sided bar construction is equivalent as a simplicial object in $\Fun(\CAlg^\mathrm{cN}_{/\kappa}, \mS)$
$$ \mathrm{\check{C}}^\bullet(\mathrm{Def}_{\w{\G}_0}\to \mathcal M_\mathrm{FG})\simeq \mathrm B^\bullet_{\mathbb G(\kappa, \w{\G}_0)}(*, \mathrm{Def}_{\w{\G}_0}) $$
to the \v{C}ech nerve of of the map $\mathrm{Def}^\mathrm{or}_{\w{\G}_0}\to \mathcal M_\mathrm{FG}$.
\end{lemma}

\begin{proof}
Unlike oriented deformations, unoriented deformations of formal groups are as a functor determined (as Kan extension) by its restriction to connective $\E$-rings by \cite[Proof of Theorem 3.4.1]{Elliptic 2}. Therefore, let us implicitly restrict all functors to the full subcategory $(\CAlg^\mathrm{cN}_{/\kappa})^\mathrm{cn}\subseteq\CAlg^\mathrm{cN}_{/\kappa}$ spanned by connective $\E$-rings for the rest of this proof.

 There, we have by \cite[Proof of Proposition 3.4.3]{Elliptic 2} a natural identification
$$\mathrm{Def}_{\w{\G}_0}\simeq (\Spec (\kappa)/\mathcal M_\mathrm{FG})_\mathrm{dR}$$
with the relative de Rham space of the morphism $\Spec(\kappa)\to\mathcal M_\mathrm{FG}$ classifying $\w{\G}_0$.
Recall from \cite[Definition 18.2.1.1]{SAG} that the \textit{relative de Rham space} of a map of functors $X\to Y$ is defined as the pullback
\begin{equation}\label{useful2}
(X/Y)_\mathrm{dR} \simeq X_\mathrm{dR}\times_{Y_\mathrm{dR}}Y,
\end{equation}
where the \textit{absolute de Rham space} of a functor $X$ is given by\footnote{Restricting to the subcategory $\CAlg^\mathrm{cN}_{/\kappa}\subseteq\CAlg^\mathrm{ad}_\mathrm{cpl}$ helps substantially here, as no colimiting over nilpotent ideals of definition is necessary.} $X_\mathrm{dR}(A) = X(\pi_0(A)/\mathfrak m)$.

Observe that we have at this point found ourselves in
the setting of Lemma \ref{Monodromy}, with the pullback square
$$
\xymatrix{
\mathrm{Def}_{\w{\G}_0}\ar[r]\ar[d] & \mathcal M_\mathrm{FG}\ar[d]\\
\Spec(\kappa)_\mathrm{dR}\ar[r] &(\mathcal M_\mathrm{FG})_\mathrm{dR}.
}
$$
playing the role of \eqref{The Square}. More precisely,
 we have 
\begin{itemize}
\item an ambient $\i$-topos $\Fun((\CAlg^\mathrm{cN}_{/\kappa})^\mathrm{cn}, \mS),$
\item a ``point" $\mathrm{Spec}(\kappa)_\mathrm{dR}\to (\mathcal M_{\mathrm{FG}})_\mathrm{dR}$,
\item an object $\mathcal M_{\mathrm{FG}}$ over  the ``base space" $(\mathcal M_\mathrm{FG})_\mathrm{dR}$, 
\item and its ``fiber" $\Spec(\kappa)_\mathrm{dR}\times_{(\mathcal M_\mathrm{FG})_\mathrm{dR}}\mathcal M_\mathrm{FG}\simeq \mathrm{Def}_{\w{\G}_0}$.
\end{itemize}
Lemma \ref{Monodromy}, \ref{un} thus exhibits the ``based loop space", which is  the de Rham space $\underline{\mathrm{Aut}}(\w{\G}_0)_\mathrm{dR}$ of
$$\Spec(\kappa)\times_{\mathcal M_\mathrm{FG}}\Spec(\kappa)\simeq \Omega_{\w{\G}_0}(\mathcal M_{\mathrm{FG}})\simeq \underline{\mathrm{Aut}}(\w{\G}_0),$$
the automorphism group  of the formal group $\w{\G}_0$, as a group object in the overcategory $\Fun((\CAlg^\mathrm{cN}_{/\kappa})^\mathrm{cn}, \mS)_{/\Spec (\kappa)_\mathrm{dR}}$. Thus Lemma \ref{Monodromy}, \ref{deux}  equips the ``fiber" $\mathrm{Def}_{\w{\G}_0}$ with the ``monodromy" $\underline{\mathrm{Aut}}(\w{\G}_0)$-action over $\Spec(\kappa)_\mathrm{dR}$, whose bar construction is
$$
\mathrm B^\bullet_{\underline{\mathrm{Aut}}(\w{\G}_0)_\mathrm{dR}}(\Spec(\kappa)_\mathrm{dR}, \mathrm{Def}_{\w{\G}_0})\simeq \mathrm{\check C}^\bullet(\mathrm{Def}_{\w{\G}_0}\to \mathcal M_\mathrm{FG}).
$$
In light of  Lemma \ref{group is group}, this $\underline{\mathrm{Aut}}({\w{\G}_0})_\mathrm{dR}$-action on the deformation (pre)stack $\mathrm{Def}_{\w{\G}_0}$ in  the overcategory $\Fun((\CAlg^\mathrm{cN}_{/\kappa})^\mathrm{cn}, \mS)_{/\Spec (\kappa)_\mathrm{dR}}$ is equivalent to a $\mathbb G(\kappa, \w{\G}_0)$-action on it in $\Fun((\CAlg^\mathrm{cN}_{/\kappa})^\mathrm{cn}, \mS)$, exhibited on the level of 
 bar constructions (see Remark \ref{Captain Planet}) by the equivalence
\begin{equation}\label{Barista}
\mathrm B^\bullet_{\underline{\mathrm{Aut}}(\w{\G}_0)_\mathrm{dR}}(\Spec(\kappa)_\mathrm{dR}, \mathrm{Def}_{\w{\G}_0})
\simeq
\mathrm B^\bullet_{\mathbb G(\kappa, \w{\G}_0)}(*, \mathrm{Def}_{\w{\G}_0}),
\end{equation}
and finishing the proof.
\end{proof}

\begin{remark}\label{Captain Planet}
We must clarify that the two bar constructions appearing on each side of the equivalence \eqref{Barista} are formed in different $\i$-categories. That is to say, the products comprising the simplices on the left-hand-side are all taken over $\Spec(\kappa)_\mathrm{dR}$, while on the right-hand-side, the products are absolute, i.e.\ taken over the terminal object $*$.
\end{remark}

\begin{remark}The de Rham space $\Spec(\kappa)_\mathrm{dR}$ that we encountered above in the proof of Lemma \ref{Action} is equivalent to the affine formal scheme $\Spf (W^+(\kappa))$, where $W^+(\kappa)$ the $\E$-ring of \textit{spherical Witt vectors over $\kappa$}, as defined in \cite[Example 5.2.7]{Elliptic 2}. Indeed, in \cite[Proof of Theorem 5.2.5]{Elliptic 2} the spherical Witt vectors are defined to corepresent as an affine formal scheme the relative de Rham space $(\Spec(\kappa)/\Spec(S))_\mathrm{dR}$. But since clearly $\Spec(S)\ \simeq \Spec(S)_\mathrm{dR}$, it follows that $(\Spec(\kappa)/\Spec(S))_\mathrm{dR}\simeq \Spec(\kappa)_\mathrm{dR}$ as claimed. More concretely, the universal property of the spherical Witt vectors may be written as
$$
\Map^\mathrm{cont}_\CAlg (W^+(\kappa), A)\simeq \varinjlim_{I} \Hom_{\CAlg^\heart}(\kappa, \pi_0(A)/I)
$$
for any adic $\E$-ring $A$, and with the colimit ranging over all of the finitely generated ideals of definition in $\pi_0(A)$. Another characterization of it is that $W^+(\kappa)$ is a flat $p$-complete $\E$-ring and $\pi_0(W^+(\kappa)) = W(\kappa)$ recovers the usual ring of ($p$-typical) Witt vectors.
\end{remark}

\begin{lemma}\label{group is group}
There is a canonical equivalence $\underline{\mathrm{Aut}}(\w{\G}_0)_\mathrm{dR}\simeq \mathbb G(\kappa, \w{\G}_0)\times \Spf (W^+(\kappa))$ of group objects in $\Fun((\CAlg^\mathrm{cN}_{/\kappa})^\mathrm{cn}, \mS)_{/\Spf (W^+(\kappa))}$.
\end{lemma}

\begin{proof}
By unwinding the definitions, we find for any connective $A\in \CAlg_{/\kappa}^\mathrm{cN}$ that
$$\underline{\mathrm{Aut}}(\w{\G}_0)_\mathrm{dR}(A)\simeq \Spec (\kappa)(\pi_0(A))/\mathfrak m)\times_{\mathcal M_\mathrm{FG}(\pi_0(A)/\mathfrak m)}\Spec(\kappa)(\pi_0(A)/\mathfrak m)
$$
consists of a pair of maps $f_1, f_2 : \kappa\to \pi_0(A)/\mathfrak m$ and an isomorphism $\varphi : f_1^*\w{\G}_0\to f_2^*\w{\G}$ of formal groups over $\pi_0(A)/\mathfrak m$. In particular, it is a discrete space. The functor $\underline{\mathrm{Aut}}(\w{\G}_0)_\mathrm{dR}$ therefore factor through  $\Set\hookrightarrow \mS $ on the left, and through $\pi_0 :(\CAlg^\mathrm{cN})^\mathrm{cn}_{/\kappa}\to(\CAlg^\heart)^\mathrm{cN}_{/\kappa}$ on the right. The same holds for $\mathbb G(\kappa, \w{\G}_0)\times \Spec (\kappa)_\mathrm{dR}$ by definition of the de Rham space. Hence it suffices to exhibit an isomorphism between these two functors in the ordinary category $\Fun((\CAlg^\mathrm{cN}_{/\kappa})^\mathrm{cn}, \Set)_{/\Spec(\kappa)_\mathrm{dR}}$.

Thus fix a complete Noetherian local commutative ring $A$ with residue field $\kappa$. Recall from the above discussion that elements of $\underline{\mathrm{Aut}}(\w{\G}_0)_\mathrm{dR}(A)$ consist of triples $(f_1, f_2, \varphi)$. Fixing $f_1\in \Spec(\kappa)_\mathrm{dR}(A)$ (since we wish to work over $\Spf (W^+(\kappa))\simeq\Spec(\kappa)_\mathrm{dR}$), we obtain an element
$$
(g, \psi)\in \mathrm{Gal}(\kappa/\mathbf F_p)\ltimes \mathrm{Aut}_{\mathrm{FGrp}(\kappa)}(\w{\G}_0)= \mathbb G(\kappa, \w{\G}_0)
$$
as follows.
Thanks to the hypothesis that $A/\mathfrak m \simeq \kappa$, the field map $f_1$ may be abstractly identified identified with a field endomorphism of $\kappa=\overline{\mathbf F}_p$. Any such endomorphism must fix the prime subfield $\mathbf F_p$, and since the inclusion $\mathbf F_p\subseteq \overline{\mathbf F}_p$ is algebraic, this implies that it is actually an automorphism. It follows that $f_1 :\kappa\to A/\mathfrak m$ is a field isomorphism, so we can set $g:= f_1^{-1}\circ f_2$. We obtain the formal group isomorphism over $\kappa$ as
$$
\psi : \w{\G}_0\xrightarrow{\varphi} (f_1^{-1})^*f_2^*\w{\G}_0\simeq g^*\w{\G}_0
$$
It is clear from the description that this procedure is bijective, compatible with the group structure, and functorial in $A$, and hence gives rise to an equivalence as claimed.
\end{proof}

\begin{remark}\label{continuous cochains}
The matter of viewing $\mathbb G(\kappa, \w{\G}_0)$ as a profinite group scheme here comes from the classical observation that that topology coincides with the usual Zariski topology on automorphisms. Indeed, as we noted in the proof, all the functors involved in Lemma \ref{group is group} factor through the functor $\pi_0 : \CAlg\to\CAlg^\heart$, and are as such a matter of classical algebraic geometry. In that context, see  \cite[Theorem 7.18]{Goerss}, or \cite[Lecture 19]{Lurie Chromatic}.

On the other hand, let us explain where the profinite structure on $\mathbb G(\kappa, \w{\G}_0)$ \textit{is coming from} from the algebro-geometric perspective.
Let us view the fixed formal group as a functor $\w{\G}_0 : (\CAlg^\mathrm{Art}_{/\kappa})^\heart\to \Set$ from Artinian local rings with residue field $\kappa$ (i.e.\ infinitesimal extensions of the point $\Spec (\kappa)$). 
Consider the subcategory $\mathrm{Nil}^{\le n}_{/\kappa}\subseteq(\CAlg^\mathrm{Art}_{/\kappa})^\heart$ of local Artinian rings with $\mathfrak m^{n+1}=0$. Restriction and Kan extension back along this inclusion produces a functor $\w{\G}{}_0^{\le n}:  (\CAlg^\mathrm{Art}_{/\kappa})^\heart\to \Set$, which we may view as the $n$-th infinitesimal neighborhood; Goerss calls this the \textit{$n$-bud} of the formal group $\w{\G}_0$, see in particular \cite[3.24 Remark]{Goerss}. Since every ideal in an Artinian local ring is nilpotent, the tower
$$
\mathrm{Nil}^{\le 0}_{/\kappa}\subseteq \mathrm{Nil}^{\le 1}_{/\kappa}\subseteq \mathrm{Nil}^{\le 2}_{/\kappa}\subseteq \mathrm{Nil}^{\le 3}_{/\kappa}\subseteq
\cdots
\subseteq(\CAlg^\mathrm{Art}_{/\kappa})^\heart
$$
is exhaustive and the canonical map $\varinjlim\w{\G}{}_0^{\le n}\to\w{\G}_0$ is an equivalence. Furthermore, any morphism of formal groups $\w{\G}\to \w{\G}'$ induces a family of maps $\w{\G}{}^{\le n}\to \w{\G}{}'^{\le n}$ for all $n\ge 0$, which induces an isomorphism
$$
\mathbb G(\kappa, \w{\G}_0) =\mathrm{Aut}
(\w{\G}_0)\simeq \varprojlim\mathrm{Aut}(\w{\G}{}^{\le n}_0).
$$
Each factor in this filtered limit is finite, recovering the usual profinite structure on the Morava stabilizer group. In particular, this implies that the product
$$
\mathbb G(\kappa, \w{\G}_0) \times \mathrm{Def}^\mathrm{or}_{\w{\G}_0}\simeq \Zpf (C^*_\mathrm{cont}(\mathbb G(\kappa, \w{\G}_0); E(\kappa,\w{\G}_0)))
$$
is the formal spectrum of an incarnation of continuous $E(\kappa, \w{\G}_0)$-valued cochains on the profinite group $\mathbb G (\kappa, \w{\G}_0)$.
\end{remark}

\begin{remark}
Let us indicate an alternative approach to proving Proposition \ref{King of Kings}. Instead of using the identification \eqref{useful}, we 
can rather observe that we have for any $A\in\CAlg^\mathrm{cN}_{/\kappa}$ a natural equivalence
$$
\mathrm{Def}_{\w{\G}_0}^\mathrm{or}(A)\simeq \mathrm{Def}_{\w{\G}_0}(\pi_0(A))\times_{\mathcal M_\mathrm{FG}(\pi_0(A))}\mathcal M_\mathrm{FG}^\mathrm{or}(A).
$$
In light of that, it suffices to establish an appropriate $\mathbb G(\kappa, \w{\G}_0)$-action on unoriented deformations,
when all functors in sight are post-composed with the functor $A\mapsto \pi_0(A)$. That involves only classical (i.e.\ non-spectral) algebraic geometry, and as such
 avoids coherence issues. Therefore the desired bar construction claim follows inductively from finding an appropriately equivariant equivalence
\begin{equation}\label{bifactor}
\mathrm{Def}_{\w{\G}_0}\times_{\mathcal M_\mathrm{FG}}\mathrm{Def}_{\w{\G}_0}\simeq \mathbb G(\kappa, \w{\G}_0)\times\mathrm{Def}_{\w{\G}_0},
\end{equation}
with both sides restricted to the subcategory $(\CAlg^\mathrm{cN}_{/\kappa})^\heart\subseteq \CAlg^\mathrm{cN}_{/\kappa}$ of discrete objects.
Since any complete Noetherian local ring may be written as a filtered limit of Artinian ones, and we are working in the ``continuous" category, it  further suffices to prove the result upon the further restriction to local Artinian rings (see \cite[Remark 7.3]{Goerss}). For that, we can reference \cite[Theorem 7.18]{Goerss}. 

There is one final small hitch: Goerss' analogue of \eqref{bifactor} takes the fiber product over a moduli functor $\widehat{\CMcal H}(n) =(\mathcal M_\mathrm{FG}^{=n}/\mathcal M_\mathrm{FG}^\mathrm{\le n})_\mathrm{dR}$ instead of over $\mathcal M_\mathrm{FG}$. But since the forgetful functor $\mathrm{Def}_{\w{\G}_0}\to \mathcal M_\mathrm{FG}$ naturally factors through this substack, that makes no difference.
\end{remark}

\subsection{The Devinatz-Hopkins Theorem}\label{Proof's home}
As before,  fix $\kappa = \overline{\mathbf F}_p$ and let $\w{\G}_0$ be of height $n$, (which specifies it up to isomorphism). We denote by $E_n$ and $\mathbb G_n$ the associated Lubin-Tate spectrum and Morava stabilizer group respectively.
Proposition \ref{King of Kings} equips $\mathrm{Def}^\mathrm{or}_{\w{\G}_0}\simeq\Spf(E_n)$ with an action\footnote{Of course this is just the action of $\mathbb G_n$ on $\Spf (E_n)$ induced by the identification of the Morava stabilizer group as $\mathbb G_n\simeq \mathrm{Aut}(E_n)\simeq\mathrm{Aut}(\mathrm{Spf}(E_n)),$ as observed in \cite[Remark 5.0.8]{Elliptic 2}. But from our way of obtaining it, its bar construction is more transparent.}  of $\mathbb G_n$ on $\mathrm{Def}^\mathrm{or}_{\w{\G}_0}\simeq\Spf(E_n)$ in the $\i$-topos $\Fun(\CAlg^\mathrm{cN}_{/\kappa})$. Let $\Spf (E_n)/\mathbb G_n$ denote the quotient of this action in this $\i$-topos. We view its ring of functions
$$
E^{h\mathbb G_n}_n := \sO(\Spf(E_n)/\mathbb G_n)
$$
 as the \textit{continuous homotopy fixed points} of the corresponding action of $\mathbb G_n$ on the Lubin-Tate spectrum $E_n$. See Remark \ref{Justicia} for some further justification of this terminology. 

\begin{theorem}[Devinatz-Hopkins]\label{DH}
With continuous homotopy fixed points defined as above,
the initial map $L_{K(n)}S \to E_n$ in $L_{K(n)}\Sp$ induces an equivalence $E^{h\mathbb G_n}_n\simeq L_{K(n)}S.$
\end{theorem}

\begin{proof}
By definition of the ring of functions, we have $\sO(\Spf (E_n))\simeq E_n$. Similarly, for products we have $\sO(\Spf(E_n)^{\times\bullet})\simeq E_n^{\widehat\otimes \bullet}$, where $\widehat{\otimes}$ denotes the completed smash product of \cite[Corollary 7.3.5.2]{HA}, i.e.\ the coproduct in the $\i$-category $\CAlg^\mathrm{ad}_\mathrm{cpl}$ of complete adic $\E$-rings. Therefore Proposition \ref{King of Kings} implies that
$$E_n^{h\mathbb G_n} \simeq \sO\big(\Zpf(E_n)^{\times(\bullet+1)}\big)
\simeq
 \Tot\big(E_n^{\widehat\otimes (\bullet+1)}\big)$$
It follows from \cite[Corollary 4.5.4]{Elliptic 2} that completion in the $\infty$-category of $E_n$-modules coincides with $K(n)$-localization, and so $E_n^{\widehat\otimes\bullet}\simeq L_{K(n)}(E_n^{\otimes\bullet})$. Thus it suffices to show that $L_{K(n)}S \to E_n$ induces an equivalence
\begin{equation}\label{uhat ue do in the shadous}
\Tot\big( L_{K(n)}\big(E_n^{\otimes (\bullet+1)}\big)\big)\simeq L_{K(n)} S.
\end{equation}
That is a standard result, stemming from the nilpotence of $L_{K(n)}S$ in the $\i$-category $L_{K(n)}\Mod_{E_n}$, and ultimately, the horizontal vanishing line in the $K(n)$-local Adams spectral sequence for $E_n$, see for instance \cite[Proposition AI.3]{DevHop}. But for completeness, we sketch an argument  anyway, following the account \cite{Mathew}, though.

The Smashing Product Theorem of Hopkins-Ravenel \cite[Theorem 7.5.6]{Ravenel} asserts that the Bousfield localization functor $L_n := L_{E_n}$ is smashing, which is by \cite[Proposition 8.2.4]{Ravenel} equivalent to  $L_nS$ being $E_n$-nilpotent. That is further equivalent, by standard nilpotence technology, e.g.\, \cite[Lectures 30 \& 31]{Lurie Chromatic}, to the cosimplicial object $(E^{\otimes (\bullet+1)}_n)$, whose totalization is $L_nS$, being pro-constant. Applying the  functor $L_{K(n)}$ to this cosimplicial object then gives the desired equivalence.
\end{proof}

\begin{remark}
An explicit analysis of how the horizontal vanishing line in the $K(n)$-local Adams spectral sequence for $E_n$ gives rise to the equivalence \eqref{uhat ue do in the shadous} is given in \cite[Section 4 and Appendix I]{DevHop}. The argument  that we gave, following \cite{Mathew}, while phrased slightly differently, is merely a re-packaging of the same fundamental idea - indeed, the proof of the Hopkins-Ravenel Smashing Product Theorem is based on the existence of a uniform vanishing line, see \cite[Section 3.4]{Math3000} for a sketch and relationship to the ``standard nilpotence technology" referred to in the proof above.
\end{remark}

\begin{remark}
The  equivalence of Theorem \ref{DH} is a purely function-level statement. Indeed, the quotient 
$\Zpf (E_n)/\mathbb G_n$ is not equivalent to the affine formal scheme $\Zpf (L_{K(n)}S)$. The value of $\Zpf (E_n)/\mathbb G_n$ on any non-complex-periodic $K(n)$-local $\E$-ring is the empty set, while the value of $\Zpf (L_{K(n)}S)$ is contractible for all $K(n)$-local $\E$-rings.
\end{remark}

Despite the preceding Remark, we may view quasi-coherent sheaves on the quotient $\Spf (E_n)/\mathbb G_n,$ which are by definition a derived version of Morava modules, as a natural incarnation in spectral formal algebraic geometry of the $K(n)$-local stable category.

\begin{corollary}\label{karfiola}
There is a canonical equivalence of symmetric monoidal $\i$-categories $$\QCoh(\Spf (E_n)/\mathbb G_n)\simeq L_{K(n)}\Sp. $$
\end{corollary}

\begin{proof}
It follows from the proof of Theorem \ref{DH} that
$$\QCoh(\Spf (E_n)/\mathbb G_n)\simeq  \Tot \big(\Mod^\mathrm{cpl}_{E_n^{\widehat{\otimes} (\bullet+1)}}\big)\simeq  \Tot\Big(L_{K(n)}\Mod_{ L_{K(n)}\left(E_n^{\otimes (\bullet+1)}\right)}\Big),$$
which is equivalent to  the $K(n)$-local stable $\i$-category in \cite[Proposition 10.10]{Mathew}.
\end{proof}

\subsection{Analogue over a general base}
At the cost of replacing the Morava stabilizer group with the more involved algebro-geometric group $\mathscr G: = \underline{\mathrm{Aut}}(\w{\G}_0)_\mathrm{dR}$, the contents of this Section  still hold after dropping the assumption that $\kappa=\overline{\mathbf F}_p$.

\begin{prop}\label{Queen of Queens}
Let formal group $\w{\G}_0$ of finite height over a perfect field $\kappa$ of positive characteristic.
There exists a canonical  $\mathscr G$-action on  $\mathrm{Def}^\mathrm{or}_{\w{\G}_0}$ in $\mathrm{Fun}(\CAlg^\mathrm{cN}_{/\kappa}, \mS)_{/\Spec(W^+(\kappa))}$, whose two-sided bar construction in said $\i$-category is equivalent as a simplicial object in the $\infty$-category $\Fun(\CAlg^\mathrm{cN}_{/\kappa}, \mS)$
$$ \mathrm{\check{C}}^\bullet(\mathrm{Def}^\mathrm{or}_{\w{\G}_0}\to *)\simeq \mathrm B^\bullet_{\mathscr G}\big(\Spf(W^+(\kappa)), \mathrm{Def}^\mathrm{or}_{\w{\G}_0}\big) $$
to the \v{C}ech nerve of (the terminal map of) $\mathrm{Def}^\mathrm{or}_{\w{\G}_0}$.
\end{prop}

\begin{proof}
The only step of the proof of Proposition \ref{King of Kings} that employs the assumption $\kappa =\overline{\mathbf F}_p$ is in the proof of Lemma \ref{group is group}.
The rest of the argument, including the proof of Lemma \ref{Action}, goes through for any perfect field $\kappa$ of positive characteristic, giving the stated result.
\end{proof}

 Proposition \ref{Queen of Queens} equips $\Spf(E(\kappa, \w {\G}))$ with a $\mathscr G$-action, though this time we need to be working in the relative setting over $\Spf (W^+(\kappa))$.
This may be viewed  as an incarnation of a $\mathscr G$-action on the Lubin-Tate spectrum $E(\kappa, \w{\G}_0 )$ in the $\infty$-category
$\CAlg^{\mathrm{ad}}_{/W^+(\kappa)}$.
Just as before, we  obtain a workable definition of continuous homotopy fixed points by setting
$$
E(\kappa, \w{\G}_0 )^{h\mathscr G}:= \sO\left(\Spf(E(\kappa, \w{\G}_0 ))/\mathscr G\right),
$$
and 
the  analogue of the Devinatz-Hopkins Theorem holds as follows.

\begin{prop}\label{DHL}
Let formal group $\w{\G}_0$ of height $n<\infty$ over a perfect field $\kappa$ of positive characteristic. 
With notation as above,  the initial map $L_{K(n)}S \to E(\kappa, \w{\G}_0)$ in $L_{K(n)}\Sp$ induces an equivalence of spectra
$
E(\kappa, \w{\G}_0 )^{h\mathscr G}\simeq L_{K(n)}S.
$
\end{prop}

\begin{proof}
The proof of Theorem \ref{DH} works just as well in this setting, provided we use \cite[ Proposition 5.2.6]{Ambi} for the nilpotence claim.
\end{proof}

\begin{remark}
As explained in \cite[Notation 2.1.10]{Ambi}, every Lubin-Tate spectrum $E(\kappa, \w{\G}_0)$ gives rise to a Morava $K$-theory $K(\kappa, \w{\G}_0)$. It might seem like we should have used the localization functor $L_{K(\kappa, \w{\G}_0 )}$ in Proposition \ref{DHL}. Alas this does not matter, since even though the 
 spectra $K(\kappa, \w{\G}_0 )$ do depend on the base-field and formal group used to define them
, the induced localization functor does not. By \cite[Remark 2.1.14]{Ambi}, the Bousfield localization functor $L_{K(\kappa, \w{\G}_0 )}\simeq L_{K(n)}$ only depends on the characteristic of the field $\kappa$ and the height $n$ of the formal group $\w{\G}_0$.
\end{remark}

In particular, we obtain by the same proof as Corollary \ref{karfiola} a ``derived Morava module" presentation of the $K(n)$-local stable category for every height $n$ formal group $\w{\G}_0$ over a perfect field $\kappa$ of positive characteristic.

\begin{corollary} Keeping all the notation from Proposition \ref{DHL},
there is a canonical equivalence of symmetric monoidal $\i$-categories $\QCoh(\Spf (E(\kappa,  \w{\G}_0)/\mathscr G)\simeq L_{K(n)}\Sp. $
\end{corollary}


\section{Spectral sequences}
The goal of this Section is to prove a version of the Morava Change of Rings Theorem, identifying the $K(n)$-local Adams spectral sequence for $E_n$ with the continuous fixed-point spectral sequence for the $\mathbb G_n$-action on $E_n$. The Devinatz-Hopkins Theorem \ref{DH} already guarantees that they converge to (filtrations on) the homotopy groups on the same spectrum, but the actual comparison of the spectral sequences (and interpretation of the second one) will take a little more work.

\subsection{The descent spectral sequence}
Unlike the fundamentally nonconnective  $\Zpf(E_n)$ and its intimidating-looking quotient $\Zpf(E_n)/\mathbb G_n$, the classifying (pre)stack $\mathrm B\mathbb G_n =*/\mathbb G_n$ is quite well-behaved. In particular, it (or better, its sheafification; but since the difference between them does not matter for quasi-coherent sheaves or functions, we will freely switch between them) is representable by a formal spectral stack which, while not quite Deligne-Mumford, is nonetheless quite manageable.

 For instance, $\QCoh(\mathrm B\mathbb G_n)$  admits an accessible $t$-structure by the (formal geometry analogue of) \cite[Proposition 6.2.5.8]{SAG}. Similarly, the descent spectral sequence,  a piece of technology familiar from the theory of topological modular forms, applies to $\mathrm B\mathbb G_n$. The following proof is essentially a repetition of the one in \cite[Chapter 5, Section 3]{TMF}, but since the setting is slightly different, we have opted to spell it out.

\begin{lemma}\label{Lemma DSS}
Let $X:(\CAlg^\mathrm{dN}_{/\kappa})^\mathrm{cn}\to\mS$ be a formal spectral fpqc stack\footnote{Here we are following \cite{SAG}, in that the absence of the adjective ``nonconnective" automatically implies connectivity.} that admits a flat cover $U\to X$, such that all the (non-trivial) fiber products $U\times_X \cdots \times_X U$ are affine formal spectral schemes. For any
 quasi-coherent sheaf $\sF$ on $X$, there exists a canonical Adams-graded spectral sequence
$$
E^{s, t}_2 = \mathrm H^s(X; \pi_t(\sF))\Rightarrow \pi_{t-s}(\Gamma(X; \sF)),
$$
called the \textit{descent spectral sequence}.
\end{lemma}

\begin{proof}[Proof]
Let $U\to X$ be a flat cover by an affine formal spectral scheme as postulated in the statement of the Lemma. It gives rise to a \v{C}ech nerve $\check{\mathrm C}^\bullet(U\to X)$ and hence a cosimplicial spectrum $\Gamma(\check{\mathrm C}^\bullet(U\to X); \sF\vert_{\check{C}^\bullet(U\to X)})$ with totalization $\Gamma( X; \sF)$. We claim that the Bousfield-Kan spectral sequence of this cosimplicial spectrum, see for instance \cite[Remark 1.2.4.4]{HA}, is the desired spectral sequence.

It converges (albeit only conditionally) to the homotopy groups of the totalization of the cosimplicial spectrum by (an opposite variant of) \cite[Proposition  1.2.2.14]{HA}. Hence it remains to show that the $E_2$ page is of the desired form. If $C^* : \Fun(\Delta, \Ab)\to \mathrm{Ch}(\Ab)^{\ge 0}$ denotes the cochain complex associated to a cosimplicial abelian group\footnote{This is the functor that participates in one direction of the Dold-Kan correspondence; see \cite[Definition 1.2.3.8]{HA} for the opposite version.}, then the second page of the Bousfield-Kan spectral sequence of a cosimplicial spectrum $M^\bullet$ may be expressed as cochain complex cohomology
$$
E_2^{s,t}=\mathrm H^s(C^*(\pi_t(M^\bullet))).
$$
To apply this to the cosimplicial object in question, we must therefore determine the homotopy groups
$$
\pi_t\Gamma(\check{\mathrm C}^\bullet(U\to X); \sF\vert_{\check{C}^\bullet(U\to X)})
\simeq
\pi_t\Gamma(U\times_X\cdots \times_X U; f^*(\sF)),
$$
where $f : U\times_X\cdots \times_X U\to X$ is the canonical map. Since $U\to X$ is flat by assumption, the same holds for $f$, and so $f^*\circ \pi_t\simeq \pi_t \circ f^*$ - see \cite[Proposition 7.2.2.13]{HA} for the affine case, from which it follows for an arbitrary flat morphism by the yoga of \cite[Subsection 6.2.5]{SAG}. Secondly, the fiber product $U\times_X\cdots \times_X U$ is affine by hypothesis, from which it follows that its global sections functor is $t$-exact. Putting all that together, we find that
$$
\pi_t\Gamma(U\times_X\cdots \times_X U; f^*(\sF))\simeq \Gamma(U\times_X\cdots \times_X U; f^*(\pi_t(\sF))),
$$
and so the $E_2$ page of the spectral sequence in question is just the standard \v{C}ech cohomology procedure for computing the $s$-th sheaf cohomology group of the quasi-coherent sheaf $\pi_t(\sF)$ on $X$.
\end{proof}

\begin{remark}
Though the approach using a cover that we sketched above will be the most convenient for us in what follows, the descent spectral sequence does not depend on that choice from the second page onwards. It may alternatively even be obtained in an invariant way: the assumptions on the stack $X$ ensure that $\QCoh(X)$ admits a well-behaved $t$-structure. Then the spectral sequence associated to the filtered object $\mathrm N(\mathbf Z)\ni n \mapsto \Gamma(\tau_{\ge -n}(X); \sF)\in \Zp$ by \cite[Definition 1.2.2.9]{HA} again gives rise to the descent spectral sequence after an appropriate re-indexing.
\end{remark}

\subsection{The Adams spectral sequence}\label{Adams section}
We wish to apply the descent spectral sequence on the quotient stack $\mathrm B\mathbb G_n$, for which we  need a quasi-coherent sheaf on it.
Consider the map $q:\Zpf (E_n)/\mathbb G_n\to\mathrm B\mathbb G_n$, induced on quotients by the terminal projection $p : \Zpf(E_n)\to *$. Using the push-forward functionality of quasi-coherent sheaves, we define the desired sheaf as $\sE_n := q_*(\sO_{\Zpf(E_n)/\mathbb G_n)})\in \QCoh(\mathrm B\mathbb G_n)$. As we will need it in the subsequent Proposition, let us identify the fiber of this quasi-coherent sheaf at the point $i :*\to */\mathbb G_n\simeq \mathrm B\mathbb G_n$. By invoking base-change along the pullback square
$$
\xymatrix{
\Zpf(E_n)\ar[r]^{\quad p} \ar[d] & {*}\ar[d]^i\\
\Zpf(E_n)/\mathbb G_n\ar[r]^{\quad q} &\mathrm B\mathbb G_n,
}
$$
we find this fiber to be
$$i^*(\sE_n)\simeq i^*q_*(\sO_{{\Zpf(E_n)}/\mathbb G_n})\simeq p_*(\sO_{\Zpf(E_n)})\simeq E_n.$$

\begin{prop}\label{ss}
The descent spectral sequence for the quasi-coherent sheaf $\sE_n$ on $\mathrm B\mathbb G_n$ is isomorphic to 
$$
E_2^{s,t}= \mathrm{Ext}^{s,t}_{\pi_*(L_{K(n)}(E_n\otimes E_n))}(\pi_*(E_n), \pi_*(E_n))\Rightarrow \pi_*(L_{K(n)}S),
$$
the $K(n)$-local Adams spectral sequence for $E_n$.
\end{prop}

\begin{proof}
Observe that both spectral sequences in question may be obtained as  Bousfield-Kan spectral sequences of certain cosimplicial spectra. Thus it suffices to exhibit an equivalence between those.

For the descent spectral sequence, we choose the flat cover $i : *\to \mathrm B\mathbb G_n$; indeed, this is a cover by the usual yoga of classifying stacks, and it is flat thanks to the Morava stabilizer group $\mathbb G_n$ being pro-\'etale and as such flat. Then the \v{C}ech nerve of $i$ is given by
$
\check{\mathrm C}^\bullet(*\to\mathrm B\mathbb G_n)\simeq \mathbb G_n^{\times \bullet},
$
and coincides with the bar construction of $\mathbb G_n$. Let $p_\bullet :\mathbb G_n^{\times\bullet}\to *$ denote the terminal map. Then it follows from the computation preceding the statement of the Proposition that
$
\sF\vert_{\check{\mathrm C}^\bullet(*\to\mathrm B\mathbb G_n)}\simeq  p_\bullet^*(E_n),
$ and so the cosimplicial spectrum that gives rise to the relevant descent spectral sequence is
$
\Gamma(\mathbb G^{\times \bullet}; p_\bullet^*(E_n)),
$
with the cosimplicial structure inherited from the bar construction of $\mathbb G_n$.

For the Adams spectral sequence,
let us apply the functor $\sO$ to the equivalence of simplicial objects of Proposition \ref{King of Kings}. We obtain an equivalence of cosimplicial spectra
$$
L_{K(n)}\big(E_n^{\otimes(\bullet +1)}\big)\simeq \sO(\mathbb G_n^\bullet\times \Zpf(E_n)).
$$
The left-hand side (for recognizing which, we have made use of a calculation from the proof of Theorem \ref{DH}), gives rise to the $K(n)$-local Adams spectral sequence for $E_n$.
To tackle the left-hand side, consider the Cartesian diagram
$$
\xymatrix{
\mathbb G_n^{\times \bullet}\times \Spf (E_n)\ar[r]^{\quad\mathrm{pr_2}} \ar[d]_{\mathrm{pr}_1}& \Spf(E_n)\ar[d]^p\\
\mathbb G_n^{\times\bullet}\ar[r]^{p_\bullet}& {*}.
}
$$
Using base-change along it, we have a series of equivalences
\begin{eqnarray*}
\sO({\mathbb G_n^{\times\bullet}\times \Spf(E_n)}) &\simeq &
\Gamma\big(\mathbb G_n^{\times\bullet}\times \Spf(E_n); \sO_{\mathbb G_n^{\times\bullet}\times \Spf(E_n)}\big) \\
&\simeq& \Gamma(\mathbb G_n^{\times \bullet} ;(\mathrm{pr}_1)_*\mathrm{pr}_2^*(\sO_{\Spf(E_n)}))\\
&\simeq& \Gamma(\mathbb G_n^{\times \bullet} ;p_\bullet^*p_*(\sO_{\Spf(E_n)}))\\
&\simeq &\Gamma(\mathbb G_n^{\times\bullet}; p_\bullet^*(E_n)),
\end{eqnarray*}
and because the cosimplicial structure comes at each step from the bar construction on $\mathbb G_n$, this is an equivalence of cosimplicial spectra.
Since we already saw that the thus-obtained cosimplicial spectrum gives rise to the descent spectral sequence for $\mathrm B\mathbb G_n$, we are done.
\end{proof}

\begin{remark}\label{ANZZ for real} By working in a non-formal setting, we may argue similarly to the above in order to obtain the Adams-Novikov spectral sequence as a special case of a descent spectral sequence (this is also explained in \cite[Remark 9.3.1.9]{SAG}).
Indeed, consider the $\E$-ring $\tau_{\ge 0}(\mathrm{MP})$, the connective cover of the periodic complex bordism spectrum. As we saw in Remark \ref{remarques}, it gives rise to a ``based loop space" $\Omega_{x_{\tau_{\ge 0}(\mathrm{MP})}}(\Spec(S))$ in spectral stacks over $\Spec(\tau_{\ge 0}(\mathrm{MP}))$. Let $\mX$ denotes the classifying (pre)stack of this spectral group scheme. Its underlying ordinary stack is given by 
$$\tau_{\le 0}(\mX)\simeq \Spec(\pi_0(\mathrm{MP}))/\Spec(\pi_0(\mathrm{MP}\otimes \mathrm{MP}))\simeq \mathcal M_\mathrm{FG}^\heart,$$
which is identified with the ordinary stack of formal groups by a celebrated theorem of Quillen. On the other hand, its (derived) ring of functions is given by
$$
\sO(\mX) \simeq  \Tot \big(\tau_{\ge 0}(\mathrm{MP})^{\otimes(\bullet +1)}\big)
= S^\wedge_{\tau_{\ge 0}(\mathrm{MP})}\simeq S,
$$
the so-called $\tau_{\ge 0}(\mathrm{MP})$-\textit{nilpotent completion}\footnote{This standard construction is often attributed to Adams and Bousfield in the literature, and sometimes just called ``derived completion", though that is unfortunate terminology, since nilpotent completion and the standard notion of derived completion only sometimes agree.} of the sphere spectrum, well-known to agree with the sphere spectrum itself. Since $\tau_{\le 0}(\mathrm{MP})\otimes \tau_{\le 0}(\mathrm{MP})$ is a flat $\tau_{\le 0}(\mathrm{MP})$-module (see \cite[Theorem 5.3.13]{Elliptic 2}), a variant of Lemma \ref{Lemma DSS} applies to the cover $\mathrm{Spec}(\tau_{\le 0}(\mathrm{MP}))\to\mX$. The resulting descent spectral sequence converges to $\pi_*(S)$, while by an argument analogous to our proof of Proposition \ref{ss}, its second page is
$$
E^{s,t}_2 =\mathrm H^s(\mathcal M_\mathrm{FG}^\heart; \pi_t(\sO_{\mX}))\simeq \mathrm{Ext}^{s,t}_{\pi_*(\tau_{\ge 0}(\mathrm{MP}\otimes\mathrm{MP}))}(\pi_*(\tau_{\ge 0}(\mathrm{MP})), \pi_*(\tau_{\ge 0}(\mathrm{MP}))),
$$
viewable either as sheaf cohomology on the underlying ordinary stack, or as (a 2-periodic variant of) the Adams-Novikov spectral sequence.
\end{remark}

\subsection{Homotopy fixed-point spectral sequence}

Let us say a few words about the interpretation of Proposition \ref{ss}.
We  may view $\QCoh(\mathrm B\mathbb G_n)$ as a version of continuous discrete representations of the Morava stabilizer group over the sphere spectrum. From that perspective, the underlying spectrum of a quasi-coherent sheaf $\sF$ on $\mathrm B\mathbb G_n$ is given by the fiber $i^*(\sF) = M$ (keeping the notation $i: *\to\mathrm B\mathbb G_n$ from the previous subsection), and the sheaf structure on $\sF$ witnesses the $\mathbb G_n$-action on $M$. The (continuous) fixed-points of this action are incarnated as global sections $M^{h\mathbb G_n} :=\Gamma( \mathrm B\mathbb G_n; \sF)$, and continuous group cohomology is given in terms of sheaf cohomology as
$$
\mathrm H^i_\mathrm{cont}(\mathbb G_n; M):=\mathrm H^i(\mathrm B\mathbb G_n; \sF)\simeq \pi_{-i}(M^{h\mathbb G_n}).
$$
Under these identifications, the descent spectral sequence for $\mathrm B\mathbb G_n$ corresponds to the fixed-point spectral sequence for $\mathbb G_n$
$$
E^{s, t}_2 =\mathrm H^s_\mathrm{cont}(\mathbb G_n;\pi_t(E_n))\Rightarrow \pi_{t-s}(E_n^{h\mathbb G_n}).
$$

\begin{remark}\label{Justicia}
In line with the preceding discussion, the sheaf $\sE_n$ on $\mathrm B\mathbb G_n$ encodes a continuous $\mathbb G_n$-action on the Lubin-Tate spectrum $E_n$. Its continuous homotopy fixed-points, in the above sense, are given by
$$
E_n^{h\mathbb G_n}\simeq \Gamma(\mathrm B\mathbb G_n; \sE_n)\simeq p_*q_*(\sO_{\Zpf(E_n)/\mathbb G_n})\simeq \Gamma(\Zpf(E_n)/\mathbb G_n; \sO_{\Zpf(E_n)/\mathbb G_n}).
$$
That agrees with (and perhaps justifies) our definition in Section \ref{Proof's home}, and its use in the Devniatz-Hopkins Theorem \ref{DH} in particular.
\end{remark}

\begin{remark}
There exist a number of precise incarnations of the $\infty$-category of continuous $\mathbb G_n$-spectra in the literature, e.g.\ of \cite{Davies} or \cite{Quick}. In each, the construction from \cite{DevHop} is enhanced (relying heavily on Devinatz-Hopkins'  detailed study of finite subgroup actions) to produce a version of $E_n$ in the respective category. Instead, we claim that
 $\QCoh(\mathrm B\mathbb G_n)$ should be viewed as an incarnation of continuous $\mathbb G_n$-spectra, sufficient for our purposes, but not intended to supplant the more sophisticated theories mentioned above (a careful comparison with which we decline to carry out).
\end{remark}

\begin{remark}
In spite of the preceding Remark, let us observe that our model at  least gives rise to spectra with a $\mathbb G_n$-action in the sense of \cite[Definition 2.2]{Determinant Tate}, referred to there as ``a simple sense of continuity". Indeed, in light of Remark \ref{continuous cochains}, a $\mathbb G_n$-action on $M$ in our sense gives rise to an augmented cosimplicial diagram $M\to C^*_\mathrm{cont}(\mathbb G_n^{\times (\bullet +1)}; M)$.
In fact, our approach to continuous $\mathbb G_n$-actions is, via the bar resolution $\mathrm B\mathbb G_n\simeq |\mathbb G_n^{\times\bullet}|$, essentially equivalent to the one of \cite{Determinant Tate}. Their restriction to the $K(n)$-local setting is mirrored in our set-up by working in the setting of formal algebraic geometry, i.e.\,inside the $\i$-category $\Fun(\CAlg^\mathrm{cN}_{/\kappa}, \mS)$ instead of  say $\Fun(\CAlg, \mS)$.
\end{remark}

\begin{remark}
With $M$ as in the previous Remark, we find by unwinding the proof of Proposition \ref{ss} that the descent spectral sequence for the corresponding sheaf on $\mathrm B\mathbb G_n$ is obtained as the Bousfield-Kan spectral sequence of the cosimplicial object $C^*_\mathrm{cont}(\mathbb G_n^{\times \bullet }; M)$. That is also one traditional approach to defining the homotopy fixed-point spectral sequence (for a compact Lie group, say),  somewhat justifying our identification of the two.
\end{remark}

With all the notation in place, the following is a formal consequence of Proposition \ref{ss}.

\begin{corollary}[Morava's Change of Rings Isomorphism]
The second page of the $K(n)$-local Adams spectral sequence of the Lubin-Tate spectrum $E_n$ may be expressed as continuous group cohomology
$
E^{s,t}_2 = \mathrm H^s_\mathrm{cont} (\mathbb G_n; \pi_t(E_n )).
$
\end{corollary}

\begin{remark}
One difference between our approach and \cite{DevHop} is that they make use of a form of Morava's Change of Rings isomorphism from \cite{DevMorava} to set up their theory. For us, on the other hand, that result did not feed into the construction of $E_n^{h\mathbb G_n}$ nor its identification with $L_{K(n)}S$, and we could instead derive it from our considerations. Of course, that is largely a cosmetic difference;  Morava's Theorem, even if classically phrased differently, ultimately boils down to algebro-geometric considerations regarding the moduli of formal groups of the sort that we based our approach on.
\end{remark}


\begin{thebibliography}{99}

\bibitem[BBGS18]{Determinant Tate}
T.~Barthel, A.~Beaudry, P.~Goerss, V.~Stojanoska, \textit{Constructing the determinant sphere using a Tate twist}, version October 15, 2018. Preprint available from \href{https://arxiv.org/abs/1810.06651}{arXiv:1810.06651 [math.AT]}

\bibitem[BD10]{Davies}
M.~Behrend, D.~Davies, \textit{The homotopy fixed point spectra of profinite Galois extensions}, Trans. Amer. Math. Soc. 362 (2010), 4983-5042. Preprint available from  \href{https://arxiv.org/abs/0808.1092}{arXiv:0808.1092 [math.AT]}


\bibitem[Dev18]{Zanath}
S.~Devalapurkar, \textit{ The Lubin-Tate stack and Gross-Hopkins duality}, version  16 July 2018. Preprint available from \href{https://arxiv.org/abs/1711.04806}{arXiv:1711.04806 [math.AT]}


\bibitem[Dev95]{DevMorava}
E.~Devinatz, \textit{Morava's change of rings theorem},  The \v{C}ech  centennial  (Boston,  MA,  1993),  Contemp.~Math., vol.~181, Amer.~Math.~Soc., Providence, RI, 1995, pp.~83–118.~MR 1320989 (96m:55007) 

\bibitem[DH04]{DevHop}
E.~Devinatz, M.~Hopkins, \textit{Homotopy fixed point spectra for closed subgroups of the Morava stabilizer groups}. Topology, 43(1): 1–47, 2004.

\bibitem[TMF]{TMF}
C.~Douglas, J.~Francis, A.~Henriques, M.~Hill, \textit{Topological Modular Forms}, volume 201.~American Mathematical Soc., 2014. Preprint available from
\url{https://math.mit.edu/events/talbot/2007/tmfproc/Chapter07/Douglas-Sheaves.pdf}

\bibitem[Goe08]{Goerss}
P.~Goerss, \textit{Quasi-coherent sheaves on the moduli stack of formal groups}, version February 7, 2008. Preprint available from \href{https://arxiv.org/abs/0802.0996}{arXiv:0802.0996 [math.AT]}

\bibitem[HL13]{Ambi}
M.~Hopkins, J.~Lurie, \textit{Ambidexterity in $K(n)$-local stable homotopy theory}, version December 19, 2013. Available from \url{https://www.math.ias.edu/~lurie/papers/Ambidexterity.pdf}


\bibitem[Lur10]{Lurie Chromatic}
J.~Lurie, \textit{Chromatic Homotopy Theory (252x)}, lecture notes from course at Harvard in  2010. Available from \url{http://people.math.harvard.edu/~lurie/252x.html}

\bibitem[Ell1]{Elliptic 1}
J.~Lurie, \textit{Elliptic Cohomology I: Spectral Abelian Varieties}, version February 3, 2018. Available from \url{https://www.math.ias.edu/~lurie/papers/Elliptic-I.pdf}


\bibitem[Ell2]{Elliptic 2}
J.~Lurie, \textit{Elliptic Cohomology II: Orientations}, version April 26, 2018. Available from \url{https://www.math.ias.edu/~lurie/papers/Elliptic-II.pdf}


\bibitem[Ell3]{Elliptic 3}
J.~Lurie, \textit{Elliptic Cohomology III: Tempered Cohomology}, version Jan 05, 2020. Available from \url{https://www.math.ias.edu/~lurie/papers/Elliptic-III-Tempered.pdf}


\bibitem[HA]{HA}
J.~Lurie, \textit{Higher Algebra}, version September 2017. Available from \url{https://www.math.ias.edu/~lurie/papers/HA.pdf}


\bibitem[HTT]{HTT}
J.~Lurie, \textit{Higher Topos Theory}, Princeton University Press, 2009. Version updated April 2017 available from \url{https://www.math.ias.edu/~lurie/papers/HTT.pdf}


\bibitem[SAG]{SAG}
J.~Lurie, \textit{Spectral Algebraic Geometry}, version February 2018. Available from \url{http://www.math.ias.edu/~lurie/papers/SAG-rootfile.pdf}




\bibitem[Mat18]{Math3000}
A.~Mathew, \textit{Examples of descent up to nilpotence}, Geometric and topological aspects of the representation theory of finite groups, 269--311, Springer Proc. Math.~Stat., 242, Springer, Cham, 2018. Preprint available from \href{https://arxiv.org/abs/1701.01528}{arXiv:1701.01528 [math.AT]}


\bibitem[Mat16]{Mathew}
A.~Mathew, \textit{The Galois group of a stable homotopy theory}, 
Adv. Math. 291 (2016), 403-541. Preprint available from \href{https://arxiv.org/abs/1404.2156}{arXiv:1404.2156 [math.CT]}

\bibitem[Qui13]{Quick}
G.~Quick, \textit{Continuous homotopy fixed points for Lubin-Tate spectra}, Homology Homotopy Appl.~Volume 15, Number 1 (2013), 191-222.
 Preprint available from \href{https://arxiv.org/abs/0911.5238}{arXiv:0911.5238 [math.AT]}


\bibitem[Rav92]{Ravenel}
D.~Ravenel, \textit{Nilpotence and periodicity in stable homotopy theory}, volume 128 of Annals of Mathematics Studies. Princeton University Press, Princeton, NJ, 1992. Appendix C by Jeff Smith.

\bibitem[Sch19]{Cond}
P.~Scholze, \textit{Lectures on Condensed Mathematics}, 2019.
Available from \url{https://www.math.uni-bonn.de/people/scholze/Condensed.pdf}


\bibitem[Tor16]{Torii}
T.~Torii. \textit{On quasi-categories of comodules and Landweber exactness}, 2016. Available from
\href{https://arxiv.org/abs/1612.03265}{arXiv:1612.03265 [math.AT]}

\end{thebibliography}
\end{document}